\documentclass[oneside,english]{amsart}
\usepackage[T1]{fontenc}
\usepackage[latin9]{inputenc}
\usepackage{amsthm}
\usepackage{amstext}
\usepackage{amssymb}
\usepackage{esint}

\makeatletter

\newcommand{\noun}[1]{\textsc{#1}}

\numberwithin{equation}{section}
\numberwithin{figure}{section}
\theoremstyle{plain}
\newtheorem{thm}{\protect\theoremname}[section]
  \theoremstyle{plain}
  \newtheorem{prop}[thm]{\protect\propositionname}
  \theoremstyle{plain}
  \newtheorem{cor}[thm]{\protect\corollaryname}
  \theoremstyle{plain}
  \newtheorem{exa}[thm]{\protect\examplename}
  \theoremstyle{plain}
  \newtheorem{lem}[thm]{\protect\lemmaname}
  \theoremstyle{definition}
  \newtheorem{defn}[thm]{\protect\definitionname}

\makeatother

\usepackage{babel}
  \providecommand{\corollaryname}{Corollary}
  \providecommand{\definitionname}{Definition}
  \providecommand{\lemmaname}{Lemma}
  \providecommand{\propositionname}{Proposition}
  \providecommand{\examplename}{Example}
\providecommand{\theoremname}{Theorem}

\begin{document}

\title{Universal Conformal Weights on Sobolev Spaces }

\author{V.Gol'dshtein and A.Ukhlov}

\begin{abstract}
The Riemann Mapping Theorem states the existence of a conformal homeomorphism $\varphi$ of a simply connected plane 
domain $\Omega\subset\mathbb C$ with non-empty boundary onto the unit disc $\mathbb D\subset \mathbb C$. 
In the first part of the paper, we study embeddings of Sobolev spaces $\overset{\circ}{W_{p}^{1}}(\Omega)$ 
into weighted Lebesgue spaces $L_{q}(\Omega,h)$ with a {}``universal'' weight that is the Jacobian of $\varphi$;
i.~e., $h(z):=J(z,\varphi)=| \varphi'(z)|^2$. Weighted Lebesgue spaces with such weights depend only on the conformal
structure of $\Omega$. For this reason, we call the weights $h(z)$ conformal weights. In the second part of the paper,
we prove compactness of embeddings of Sobolev spaces $\overset{\circ}{W_{2}^{1}}(\Omega)$ into $L_{q}(\Omega,h)$
for any $1\leq q<\infty$. With the help of Brennan's Conjecture, we extend these results to the Sobolev spaces
$\overset{\circ}{W_{p}^{1}}(\Omega)$. In this case, $q$ depends on $p$ and the integrability 
exponent for Brennan's Conjecture. The last part of the paper is devoted to applications to elliptic boundary value
problems.
\end{abstract}

\maketitle


\section{Introduction }

Let $\Omega\subset\mathbb C$ be an arbitrary simply connected
plane domain with non-empty boundary. By the Riemann Mapping Theorem, there exists a conformal homeomorphism $\varphi$ 
of  $\Omega$  onto the unit disc $\mathbb D\subset \mathbb C$.

This study is focused on the weighted Poincar\'e-Sobolev inequalities 
\begin{equation}
\biggl(\int\limits _{\Omega}|f(z)|^{r}h(z)\, d\mu \biggr)^{\frac{1}{r}}\leq K\biggl(\int\limits _{\Omega}
|\nabla f(z)|^{p}d\mu\biggr)^{\frac{1}{p}}\label{eq:WPI}
\end{equation} for functions $f$ of the Sobolev space $\overset{\circ}{W_{p}^{1}}(\Omega)$ and a special weight $h(z):=J(z,\varphi)=|\varphi^{\prime} (z)|^2$ induced by $\varphi$. Recall that  $\overset{\circ}{W_{p}^{1}}(\Omega)$ is closure of the set of all smooth functions with compact support in $\Omega$ in the Sobolev space ${W_{p}^{1}}(\Omega)$ and $J(z,\varphi)$ is Jacobian of a conformal homeomorphisms 
$\varphi: \Omega\to\mathbb D$ at $z\in\Omega$.

{\it Here we have used the following notations: $z=x+iy$ is a complex number, $f(z)=f(x,y)$ is a real-valued function, 
$\nabla f(z)=(\frac{\partial f}{\partial x},\frac{\partial f}{\partial y})$ is the weak gradient of $f$, $\mu$ is the Lebesgue measure.}
 
In the present paper, we give an essentially self-contained exposition of a method for the study of such weighted 
inequalities and develop its applications to elliptic boundary value problems. 

The novel points include the use of a "transfer" scheme of Sobolev type embedding theorems from regular domains
to non-regular domains (proposed in \cite{GGu}) in combination with the Riemann Mapping Theorem and the Brennan's 
Conjecture. The ``transfer'' scheme is based on systematic applications of the theory of composition operators
on Sobolev spaces \cite{U1, VU1}. In \cite{GU} this scheme was applied to weighted Sobolev-type embedding operators
in non-regular domains.

The Poincar\'e-Sobolev-type inequalities have essential applications
in geometric analysis (see, for example, \cite{M}). The existence of
the weighted inequalities (\ref{eq:WPI}) is interesting even for bounded
domains in the case of unbounded weights $h(z)$. For unbounded weights,
the inequalities (\ref{eq:WPI}) contain additional information about
the boundary behavior of the functions $f\in\overset{\circ}{W_{p}^{1}}(\Omega)$. 

Note, that such type inequalities have an application \cite{AN} in a regularity result for the Poisson problem $-\Delta u=f$ , $u\vert_P=g$ on a polyhedral domain 
$P\subset \mathbb R^3$. 

The ``transfer'' scheme \cite{GGu} is simplified here because of the well-known fact: For any conformal homeomorphism 
$w=\varphi(z):\Omega\to\Omega'$ and any smooth function $f$ with square integrable derivatives, we have 
$$
\int\limits _{\Omega}|\nabla(f\circ\varphi)(z))|^{2}d\mu=\int\limits _{\Omega'}| \nabla f(w)|^{2}d\mu.
$$
This equality means that $\varphi$ induces an isometry of homogeneous Sobolev spaces $L^1_2(\Omega)$ 
and  $L^1_2(\Omega')$. It is one of the basic facts that makes it possible to ``transfer''
the Poincar\'e-Sobolev inequalities from the unit disc $\mathbb{D}\subset\mathbb{C}$ to an arbitrary simply 
connected plane domain with non-empty boundary $\Omega\subset\mathbb{C}$.

In the first part of the paper, we prove existence of bounded (compact) embeddings of the~Sobolev space
$\overset{\circ}{W_{2}^{1}}(\Omega)$ into weighted Lebesgue spaces
$L_{q}(\Omega,h)$ with the conformal weight $h(z):=J(z,\varphi)$ for any $q\in [1,\infty)$. 
Since two different conformal homeomorphisms
$\varphi:\Omega\to\mathbb{D}$ and $\tilde{\varphi}:\Omega\to\mathbb{D}$
can be connected by a~conformal automorphism $\eta:\mathbb{D}\to\mathbb{D}$ 
(i.e., $\varphi=\tilde{\varphi}\circ\eta$), the conformal weights induced by $\varphi$ and $\tilde{\varphi}$ 
are equivalent. It means that $h(z)=J(z,\varphi)\lesssim J(z,\tilde{\varphi})\lesssim J(z,\varphi)=h(z)$,
and therefore the weighted Lebesgue space $L_{q}(\Omega,h)$ does not depend on the choice
of a conformal homeomorphism and depends only on the conformal structure of $\Omega$. 

In the second part of the paper, we study the more complicated case of Sobolev spaces 
$\overset{\circ}{W_{p}^{1}}(\Omega)$,  $p\neq 2$. Applying results of Brennan's Conjecture about 
the integrability of the derivatives of conformal homeomorphisms $\varphi:\Omega \to \mathbb{D}$,
we prove that such homeomorphisms induce bounded composition operators from the homogeneous Sobolev spaces 
$L_{p}^{1}(\Omega)$ into $L_{q}^{1}(\mathbb{D})$ under some constraints on $p$ and $q$ that are 
consequences of the results of Brennan's Conjecture. Another ingredient of this study is necessary and 
sufficient conditions  \cite{U1} for boundedness of the composition operators from $L_{p}^{1}(\Omega)$ 
into $L_{q}^{1}(\mathbb{D})$ induced by Sobolev homeomorphisms. This result is rather general and is rearranged 
here for the conformal case.

In the last part of the paper, we show a standard application of the main results to boundary value problems 
for the Laplace operator. 

We use a version of Brennan's Conjecture for composition operators on Sobolev spaces \cite{GU3} proposed recently 
by the authors. The original Brennan's Conjecture concerns the integrability of derivatives of plane
conformal homeomorphisms $\varphi:\Omega\to\mathbb{D}$ that map a simply
connected plane domain with non-empty boundary $\Omega\subset\mathbb{C}$
onto the unit disc $\mathbb{D}\subset\mathbb{C}$.

The conjecture
\cite{Br} is that 
\begin{equation}
\int\limits _{\Omega}|\varphi'(z)|^{s}~d\mu<+\infty,\quad\text{for all}\quad\frac{4}{3}<s<4.\label{eq:BR}
\end{equation}

For $4/3<s<3$, it is a comparatively easy consequence of the Koebe distortion theorem (see, for example, \cite{Ber}).
J.~Brennan \cite{Br} (1973) extended this range to $4/3<s<3+\delta$, where $\delta>0$, and conjectured it to hold 
for $4/3<s<4$. The example of
$\Omega=\mathbb{C}\setminus(-\infty,-1/4]$ shows that this range
of $s$ cannot be extended. The upper bound of those $s$ for which
(\ref{eq:BR}) is known to hold has been increased to $s\leq3.399$ by Ch. Pommerenke,
to $s\leq3.421$ by D. Bertilsson, and then to $s\leq 3.752$ by Hedenmalm and Shimorin (2005).  These
results and more information can be found in \cite{Ber,Pom, Shim}.

For conformal homeomorphisms $\psi:\mathbb{D}\to\Omega$, Brennan's
Conjecture can be reformulated as the Inverse Brennan's Conjecture 
\[
\int\limits _{\mathbb{D}}|\psi'(w)|^{\alpha}~d\mu<+\infty,\quad\text{for all}\quad-2<\alpha<2/3
\]
where $\alpha=2-s$. 

The results of Inverse Brennan's Conjecture lead to the conjecture on the existence
of bounded composition operators of $\overset{\circ}{W_{p}^{1}}(\Omega)$ to $\overset{\circ}{W_{q}^{1}}(\mathbb D)$
for all $4/3<p<2$ and all $1\leq q<2p/(4-p)$  (Theorem \ref{thm:InverseCompL}).  As a corollary, we obtain a conjecture about
the existence of compact embeddings of $\overset{\circ}{W_{p}^{1}}(\Omega)$ into $L_{r}(\Omega,h)$
for all 
$$
 1\leq r<\frac{p}{2-p}.
$$
If $\alpha_{0}>-2$ is the best known estimate in the Inverse Brennan's 
Conjecture, i.~e. the Inverse Brennan's Conjecture holds for any $\alpha\in [\alpha_0, \frac{2}{3})$, then 
$$
1\leq r\leq \frac{2p}{2-p}\cdot\frac{\left|\alpha_{0}\right|}{2+\left|\alpha_{0}\right|}<\frac{p}{2-p}
$$
is the best estimate for $r$ for these embeddings \cite{GU3}.

A connection between Brennan's Conjecture and composition operators was established in \cite{GU3}:

\begin{thm}{\bf Equivalence Theorem.} 
Brennan's Conjecture (\ref{eq:BR}) holds for a number $s\in ({4}/{3};4)$ if and only if any conformal 
homeomorphism $\varphi : \Omega\to\mathbb D$ induces a bounded composition operator
$$
\varphi^{\ast}: L^1_{p}(\mathbb D)\to L^1_{q(p,s)}(\Omega)
$$
for any $p\in (2;+\infty)$ and $q(p,s)=ps/(p+s-2)$.
\end{thm}

{\bf Remark.} Brennan's Conjecture is correct for some special classes of domains: starlike domains, 
bounded domains which boundaries are locally graphs of continuous functions etc.

\section{Notation and Preliminary Results about Composition Operators}

We follow \cite{HKM} for notation and basic facts about weighted
Lebesgue spaces.

Let $\Omega\subset\mathbb{R}^{n}$ be a domain and let $v:\Omega\to \mathbb R$
be a locally integrable almost everywhere positive real valued function in $\Omega$
( i.e $v(z)>0$ almost everywhere). Then a Radon measure $\nu$ is
canonically associated with the weight function $v$:
\[
\nu(E):=\int_{E}v(z)d\mu.
\]

By the local integrability of $v$, the measure $\nu$ and the Lebesgue
measure $\mu$ are absolutely continuous with respect to one another:
\[
d\nu=v(z)d\mu.  
\]
 In what follows, the weight $v$ and the measure $\nu$ will be identified.
The sets of measure zero are the same for the Lebesgue measure $\mu$ and for~$\nu$.
It means that we do not need to specify what convergence almost everywhere is. 

Denote by 
$$
\mathcal{V}(\Omega):=\{v\in L_{1, loc}(\Omega):v(z)>0\,\, \text{a.~e. on}\,\, \Omega \}
$$ the set of all such weights. Here $L_{1, loc}(\Omega)$
is the space of locally integrable functions in $\Omega$.

For $1\leq p<\infty$ and $v\in\mathcal{V}(\Omega)$, consider the
weighted Lebesgue space 
\[
L_{p}(\Omega,v):=\left\{ f:\Omega\to R:\| f\mid {L_{p}(\Omega,v)}\|:=\left(\int_{\Omega}|f(z)|^{p}d\nu\right)^{1/p}<\infty \right\}.
\]
It is a Banach space for the norm $\|f\mid{L_{p}(\Omega,v)}\|$.

The space $L_{p}(\Omega,v)$ may fail to embed into $L_{1, loc}(\Omega)$. 

\begin{prop}
\cite{HKM}\label{WL} If $v^{\frac{1}{1-p}}\in L_{1,loc}(\Omega)$
and $1< p<\infty$ then the embedding operator $i:L_{p}(\Omega,v)\to L_{1, loc}(\Omega)$
is bounded. 

If $v^{-1}\in L_{\infty, loc}(\Omega)$ then the embedding operator
$i:L_{1}(\Omega,v)\to L_{1, loc}(\Omega)$ is bounded. 
\end{prop}
For $1<p<\infty$, we put
\[
\mathit{\mathcal{V}_{p}(\Omega):=\left\{ v\in\mathcal{V}(\Omega):v^{\frac{1}{1-p}}\in L_{1, loc}(\Omega)\right\} }
\]
and for $p=1$, 
\[
\mathit{\mathcal{V}_{1}(\Omega):=\left\{ v\in\mathcal{V}(\Omega):v^{-1}\in L_{\infty, loc}(\Omega)\right\} }.
\]

\begin{cor}
If a weight $v$ is continuous and positive then $i:L_{p}(\Omega,v)\to L_{1, loc}(\Omega)$
is bounded.
\end{cor}
This follows immediately from Proposition 2.1 because a continuous and
positive weight belongs to $\mathcal{V}_{p}(\Omega)$ and also to
$\mathcal{V}_{1}(\Omega)$. 

Define the Sobolev space $W_{p}^{1}(\Omega)$, $1\leq p<\infty$,
as a normed space of locally integrable weakly differentiable functions
$f:\Omega\to\mathbb{R}$ equipped with the following norm: 
\[
\|f\mid W_{p}^{1}(\Omega)\|=\biggr(\int\limits _{\Omega}|f(z)|^{p}\, d\mu\biggr)^{1/p}+\biggr(\int\limits _{\Omega}|\nabla f(z)|^{p}\, d\mu\biggr)^{1/p}.
\]
We will also need homogeneous seminormed Sobolev spaces $L_{p}^{1}(\Omega)$
of weakly differentiable functions $f:\Omega\to\mathbb{R}$ equipped
with the following seminorms: 
\[
\|f\mid L_{p}^{1}(\Omega)\|=\biggr(\int\limits _{\Omega}|\nabla f(z)|^{p}\, d\mu\biggr)^{1/p}.
\]

Recall that the embedding operator $i:L_{p}^{1}(\Omega)\to L_{1, loc}(\Omega)$
is bounded.

The Sobolev space $\overset{\circ}{W_{p}^{1}}(\Omega)$, $1\leq p<\infty$,
is defined as the closure of the space of smooth functions with compact
supports $C_{0}^{\infty}(\Omega)$ in the norm of $W_{p}^{1}(\Omega)$.

Let $\Omega$ and $\Omega'$ be domains in $\mathbb{C}$. We say that a conformal
homeomorphism $\varphi:\Omega\to\Omega'$ induces a bounded composition
operator 
\[
\varphi^{\ast}:L_{p}^{1}(\Omega')\to L_{q}^{1}(\Omega),\,\,\,1\leq q\leq p\leq\infty,
\]
by the composition rule $\varphi^{\ast}(f)=f\circ\varphi$, if for
any $f\in L_{p}^{1}(\Omega')$, the composition $\varphi^{\ast}(f)\in L_{q}^{1}(\Omega)$
and there exists a constant $K<\infty$ such that 
\[
\|\varphi^{\ast}(f)\mid L_{q}^{1}(\Omega)\|\leq K\|f\mid L_{p}^{1}(\Omega')\|.
\]

The theory of composition operators on Sobolev spaces goes back to
Reshetnyak's problem about the description of all isomorphisms of
the~homogeneous Sobolev spaces $L_{n}^{1}(\Omega)$ and $L_{n}^{1}(\Omega')$ 
which are induced by quasiconformal
mappings of Euclidean domains. In \cite{VG1}, it was proved that 
a~homeomorphism $\varphi:\Omega\to\Omega'$ between Euclidean domains $\Omega \subset \mathbb{R}^{n}$ 
and $\Omega'\subset \mathbb{R}^{n}$ induces an isomorphism of $L_{n}^{1}(\Omega)$ and $L_{n}^{1}(\Omega')$ 
by the composition rule $\varphi^{\ast}(f)=f\circ\varphi$ if and only if $\varphi$ is quasiconformal. 

In the framework of this approach to geometric function theory, there appears 
the problem about the description of homeomorphisms inducing isomorphisms
of Sobolev spaces by the composition rule. The Sobolev spaces
$W_{p}^{1}(\Omega)$ and $W_{p}^{1}(\Omega')$, $p>n$, were considered in \cite{VG2}, the Sobolev
spaces $W_{p}^{1}(\Omega)$ and $W_{p}^{1}(\Omega')$, $n-1<p<n$, were treated in \cite{GRo}, 
and Sobolev spaces $W_{p}^{1}(\Omega)$ and $W_{p}^{1}(\Omega')$, $1\leq p<n$, were considered in \cite{Mar}. 
In \cite{MSh}, the theory of multipliers was applied to this composition problem. Bounded composition 
operators on Sobolev spaces were studied in \cite{Vod1} from another (but close) point of view. A geometric
description of homeomorphisms preserving the Sobolev spaces $L_{p}^{1}(\Omega')$
and $L_{p}^{1}(\Omega)$ was obtained in \cite{GGR} for $p>n-1$.

New problems arise when we study composition operators on Sobolev
spaces with decreasing integrability of the first weak derivatives.
This problem was first studied in~\cite{U1}. In this
case, a significant role in the description of the composition operators
is played by the so-called (quasi)additive set functions defined on open sets.

The main result of \cite{U1} asserts that 
 \begin{thm} \label{CompTh}
 A homeomorphism $\varphi:\Omega\to\Omega'$ between two domains
$\Omega,\Omega'\subset\mathbb{R}^{n}$ induces a bounded
composition operator 
\[
\varphi^{\ast}:L_{p}^{1}(\Omega')\to L_{q}^{1}(\Omega),\,\,\,1\leq q< p<\infty,
\]
 if and only if $\varphi\in W_{1,loc}^{1}(\Omega)$, has
finite distortion, and 
\[
K_{p,q}(f;\Omega)=\biggl(\int\limits _{\Omega}\biggl(\frac{|D\varphi(x)|^{p}}{|J(x,\varphi)|}\biggr)^{\frac{q}{p-q}}d\mu\biggr)^{\frac{p-q}{pq}}<\infty.
\]
\end{thm}
Here $D\varphi(x)$ is the formal Jacobi matrix of $\varphi$ at $x\in\Omega$ and $J(x,\varphi)=\det D\varphi(x)$ is its Jacobian.
The norm $|D\varphi(x)|$ of the matrix is the norm of the linear operator defined
by this matrix in the Euclidean space $\mathbb R^n$.
 In \cite{GGu}, this class of mappings was studied in connection
with the Sobolev-type embedding theorems. 

The detailed study of composition
operators on Sobolev spaces was carried out in \cite{VU1}. The composition
operators on Sobolev spaces in the limit case $p=\infty$ were studied
in \cite{GU1,GU2}.

Define the $p$-dilatation of a diffeomorphism $\varphi:\Omega\to\Omega'$ as 
$$
K_p(x,\varphi)=\frac{|D \varphi(x)|^p}{|J(x,\varphi)|}
$$
and introduce the $(p,q)$-dilatation as 
\[
K_{p,q}(f;\Omega)=\biggl(\int\limits _{\Omega}\biggl(K_p(x,\varphi)\biggr)^{\frac{q}{p-q}}d\mu\biggr)^{\frac{p-q}{pq}}.
\]

For plane conformal homeomorphisms ($n=2$), the $p$-dilatation is equal to $| \varphi'(z)|^{p-2}$ for any $ p\in[1,\infty)$. Of course, for $p=2$ the $2$-dilatation is the classical conformal dilatation ${|\varphi'(z))|^{2}}/{J(z, \varphi)}$ and is equal to the unity for conformal mappings.

\section{Weighted Poincar\'e-Sobolev Inequalities for the Sobolev Space $\protect\overset{\circ}{W_{2}^{1}}(\Omega)$}

In this section, we study the weighted Poincar\'e-type inequalities for conformal weights.

First we formulate a well known property of conformal homeomorphisms:
\begin{lem}
\label{lem:isometry}Let $\Omega$ and $\Omega'$ be two plane domains.
Any conformal homeomorphism $w=\varphi(z):\Omega\to\Omega'$ induces an
isometry of spaces $L_{2}^{1}(\Omega')$ and $L_{2}^{1}(\Omega)$.\end{lem}
\begin{proof}
Let $f\in L_{2}^{1}(\Omega')$ be a smooth function. Then the smooth function $g=f\circ\varphi$
belongs to $L_{2}^{1}(\Omega)$ because 
\begin{multline*}
\|\nabla g\mid L_{2}({\Omega})\|=\biggl(\int\limits _{{\Omega}}|\nabla(f\circ\varphi(z))|^{2}~d\mu\biggr)^{\frac{1}{2}}=\biggl(\int\limits _{{\Omega}}|\nabla f|^{2}(\varphi(z))|\varphi'(z))|^{2}~d\mu\biggr)^{\frac{1}{2}}\\
=\biggl(\int\limits _{{\Omega}}|\nabla f|^{2}(\varphi(z))J(z,\varphi)~d\mu\biggr)^{\frac{1}{2}}=\biggl(\int\limits _{\Omega'}|\nabla f|^{2}(w)~d\mu\biggr)
=\|\nabla f\mid L_{2}(\Omega')\|.
\end{multline*}

We used  the following conformal equality:  $|\varphi'(z))|^{2}=J(\varphi,z)$ for every $z\in\Omega$. Approximating an arbitrary function 
$f\in L_{2}^{1}(\Omega')$ by smooth functions, we obtain an isometry between $L_{2}^{1}(\Omega')$ 
and $L_{2}^{1}(\Omega)$.
\end{proof}

\begin{defn}
Let $\Omega\subset\mathbb{C}$ be a simply connected domain with
non-empty boundary and let $\varphi$ be a conformal homeomorphism of
$\Omega$ onto the unit disc $\mathbb{D}$. We call the smooth positive real-valued function
$h(z)=J(z,\varphi)=|\varphi^{\prime} (z)|^2$ the universal conformal weight in $\Omega$
(or simply the conformal weight).

Recall that the Lebesgue spaces $L_{p}(\Omega,h)$ does not depend on
the choice of the conformal homeomorphism $\varphi$; i.e. the dependence is only on the conformal
structure of $\Omega$. This is a reason to call the weight $h(z)$ the (universal)
conformal weight on $\Omega$.\end{defn}

\begin{thm}
\label{thm:Weightem} Let $\Omega\subset\mathbb{C}$ be a simply
connected domain with non-empty boundary. Then the inequality
$$
\|f\mid L_{r}(\Omega,h)\|\leq K\|\nabla f\mid L_{2}(\Omega)\|
$$ 
holds for every function
$f\in C_{0}^{\infty}(\Omega)$ and for any $1\leq r<\infty$. Here $K$ is a constant depending only on $r$.
\end{thm} 
{\bf Remark.} The constant $K$ is equal to the exact constant for the corresponding Poincar\'e-Sobolev inequality 
in the unit disc $\mathbb D\subset\mathbb C$, i.~e., for 
$$
\|f\mid L_{r}(\mathbb D)\|\leq K\|\nabla f\mid L_{2}(\mathbb D)\|,\,\,\,f\in C_{0}^{\infty}(\mathbb D).
$$ 
\begin{proof}
Let $f\in C_{0}^{\infty}(\Omega)$. By the Riemann Mapping Theorem, there exists a conformal homeomorphism
$\varphi: \Omega\to \mathbb D$. 
Then the function $g=f\circ\varphi^{-1}$
belongs to $C_{0}^{\infty}(\mathbb{D})$ and, by Lemma \ref{lem:isometry},
$\|\nabla g\mid L_{2}(\mathbb{D})\|=\|\nabla f\mid L_{2}(\Omega)\|$.

Using the Poincar\'e-Sobolev inequality for the function $g$ in the disc
$\mathbb{D}$ we infer 
\begin{multline*}
\|f\mid L_{r}(\Omega,h)\|=\biggl(\int\limits _{\Omega}|f(z)|^{r}J(z,\varphi)~d\mu\biggr)^{\frac{1}{r}}=\biggl(\int\limits _{\mathbb{D}}|f\circ\varphi^{-1}(w)|^{r}~d\mu\biggr)^{\frac{1}{r}}\\
=\|g\mid L_{r}(\mathbb{D})\|\leq K\| \nabla g\mid L_{2}(\mathbb{D})\|=K\|\nabla f\mid L_{2}(\Omega)\|.
\end{multline*}
 
\end{proof}

Now we give some examples of conformal weights.

\begin{exa}
\label{exa:plane}
Let $\Omega_{pl}=\mathbb C\setminus\overline{\mathbb D}=\{z\in \mathbb C: x^2+y^2>1\}$ be the plane without the unit disc. The diffeomorphism
$$
w=\varphi(z)=\frac{1}{z},\,\,\,z=x+iy,
$$
is conformal and maps $\Omega_{pl}$ onto the unit disc $\mathbb D$. The conformal weight is 
$$
h(z)=\frac{1}{|z^2|^2}=\frac{1}{(x^2+y^2)^2}.
$$ 
\end{exa}

\begin{exa}
\label{exa:half-plane}
Let $\Omega_h=\mathbb C_{+}=\{z\in \mathbb C: y>0\}$ be the upper half-plane. The diffeomorphism
$$
w=\varphi(z)=\frac{z-i}{z+i},\,\,\,z=x+iy,
$$
is conformal and maps $\Omega_h$ onto the unit disc $\mathbb D$. Then the conformal weight is
$$
h(z)=\frac{4}{|z+i|^4}=\frac{4}{(x^2+(y+1)^2)^2}.
$$ 
\end{exa}

\begin{exa}
\label{exa:strip}
Let $\Omega_s=\{z\in\mathbb C: -\frac{\pi}{4}<Re~ z<\frac{\pi}{4}\}$ be a strip. The diffeomorphism
$$
w=\varphi(z)=\frac{1}{i}\frac{e^{2iz}-1}{e^{2iz}+1}=\tan z,\,\,\,z=x+iy,
$$
is conformal and maps $\Omega$ onto the unit disc $\mathbb D$. Then the conformal weight is
$$
h(z)=\frac{1}{|z^2+1|^2}=\frac{1}{(x^2+y^2)^2+x^2-y^2+1}.
$$ 
\end{exa}

\begin{exa}
\label{exa:cardioid}
Let $\Omega_c$ be the interior of the cardioid: $r=\frac{1}{2}(1+\cos \theta)$. The diffeomorphism
$$
w=\varphi(z)=\sqrt{z}-1,\,\,\,z=x+iy,
$$
is conformal and maps $\Omega_c$ onto the unit disc $\mathbb D$. Then the conformal                                                                                                                                      weight 
$$
h(z)=\frac{1}{2\sqrt{|z|}}=\frac{1}{2\sqrt[4]{x^2+y^2}}.
$$ 
\end{exa}

\section{Embedding into Lebesgue Spaces with Conformal Weights}

In this section, we prove existence of compact embeddings
of the Sobolev spaces $\overset{\circ}{W_{p}^{1}}(\Omega)$ into the Lebesgue
spaces $L_{r}(\Omega,h)$ with the~(universal) conformal weights $h$.
We begin with a general fact about Sobolev spaces $\overset{\circ}{W_{2}^{1}}(\Omega)$.
\begin{thm}
\label{thm:Comp} Let $\varphi:\mathbb{D}\to\Omega$ be a conformal
homeomorphism. Then the composition operator
\[
\varphi^{\ast}:\overset{\circ}{W_{2}^{1}}(\Omega)\to\overset{\circ}{W_{2}^{1}}(\mathbb{D})
\]
is bounded. \end{thm}
\begin{proof}
Let  $f\in\overset{\circ}{W_{2}^{1}}(\Omega)$ be a smooth function. We first prove that $\varphi^{\ast}f$ 
belongs to $\overset{\circ}{W_{2}^{1}}(\mathbb{D})$.

By the definition of the  composition operators $\varphi^{\ast}f=f\circ\varphi$ is a smooth function defined everywhere in $\Omega$ and by Lemma \ref{lem:isometry},
we have the equality
$$
\|\nabla (\varphi^{\ast}f)|L_{2}(\mathbb{D})\|=\| \nabla f|L_{2}(\Omega)\|.
$$

Since the support $\mathrm{supp}(f)$ of~$f$ is compact, its inverse image $\varphi^{-1}(\mathrm{supp}(f))$
is also compact. Hence, we obtain the following estimate of $\|\varphi^{\ast}f|L_{2}(\mathbb{D})\|$:
\begin{multline*}
\|\varphi^{\ast}f|L_{2}(\mathbb{D})\|=\biggl(\int\limits _{\varphi^{-1}(\mathrm{supp}(f))}|f\circ\varphi|^{2}~d\mu\biggr)^{\frac{1}{2}}\\
=\biggl(\int\limits _{\varphi^{-1}(\mathrm{supp}(f))}|f\circ\varphi|^{2}J(z,\varphi)\frac{1}{J(z,\varphi)}~d\mu\biggr)^{\frac{1}{2}}\\
\leq\max\limits _{z\in\varphi^{-1}(\mathrm{supp}(f))}\biggl(\frac{1}{J(z,\varphi)^{\frac{1}{2}}}\biggr)\biggl(\int\limits _{\varphi^{-1}(\mathrm{supp}(f))}|f\circ\varphi|^{2}J(z,\varphi)~d\mu\biggr)^{\frac{1}{2}}\\
\leq\max\limits _{z\in\varphi^{-1}(\mathrm{supp}(f))}\biggl(\frac{1}{J(z,\varphi)^{\frac{1}{2}}}\biggr)\biggl(\int\limits _{\mathrm{supp}(f)}|f|^{2}(w)~d\mu\biggr)^{\frac{1}{2}}
\end{multline*}

Using the notation
$$
Q(\varphi,f):=\max\limits _{z\in\varphi^{-1}(\mathrm{supp}(f))}\biggl(\frac{1}{J(z,\varphi)^{\frac{1}{2}}}\biggr).
$$
we conclude that
$$
\|\varphi^{\ast}f|L_{2}(\mathbb{D})\| \leq Q(\varphi,f) \|f|L_{2}(\Omega)\|.
$$
So, the composition $\varphi^{\ast}f$ belongs to the Sobolev space
$\overset{\circ}{W_{2}^{1}}(\mathbb{D})$.

Let us prove that the composition operator $\varphi^{\ast}$ is bounded.

By the Poincar\'e inequality 
\[
\|\varphi^{\ast}f|L_{2}(\mathbb{D})\|\leq C\|\nabla (\varphi^{\ast}f)|L_{2}(\mathbb{D})\|
\]
for every function $g=\varphi^{\ast}f\in \overset{\circ}{W_{2}^{1}}(\mathbb{D})$, here the constant $C$ does not depends on $f$.
 
Hence, 
\begin{multline*}
\|\varphi^{\ast}f|\overset{\circ}{W_{2}^{1}}(\mathbb{D})\|=\|\varphi^{\ast}f|L_{2}(\mathbb{D})\|+\|\nabla(\varphi^{\ast}f)|L_{2}(\mathbb{D})\|\\
\leq C\|\nabla(\varphi^{\ast}f)|L_{2}(\mathbb{D})\|+
\|\nabla(\varphi^{\ast}f)|L_{2}(\mathbb{D})\|\\
=(C+1)\| \nabla f|L_{2}(\Omega)\|\leq (C+1)\|f|\overset{\circ}{W_{2}^{1}}(\Omega)\|
\end{multline*}
for every smooth function  $f\in \overset{\circ}{W_{2}^{1}}(\Omega)$. 

Using the density of smooth functions with compact supports in $\overset{\circ}{W_{2}^{1}}(\Omega)$,
we can extend the last inequality to an arbitrary function $f\in \overset{\circ}{W_{2}^{1}}(\Omega)$ 
(see, for example, \cite{VU1, GMU}). 
This means that $\varphi^{\ast}f\in\overset{\circ}{ W^1_2}(\mathbb D)$ and the composition operator 
\[
\varphi^{\ast}:\overset{\circ}{W_{2}^{1}}(\Omega)\to\overset{\circ}{W_{2}^{1}}(\mathbb{D})
\]
 is bounded. 
\end{proof}

The previous theorem leads to compactness of the Sobolev type embeddings
 \\ $j_r:\overset{\circ}{W_{2}^{1}}(\Omega)\hookrightarrow L_{r}(\Omega,h)$ in the case 
of the universal conformal weight $h$:

\begin{thm}
Let $\Omega\subset\mathbb{C}$ be a simply connected domain with
non-empty boundary. Then the embedding operator 
\[
j_r:\overset{\circ}{W_{2}^{1}}(\Omega)\hookrightarrow L_{r}(\Omega,h)
\]
is compact for any $1\leq r<\infty$. Here $h$ is the universal conformal
weight.\end{thm}
\begin{proof}
By the Riemann Mapping Theorem, there exists a conformal homeomorphism
$w=\varphi(z):\Omega\to\mathbb{D}$. The inverse mapping is also conformal
and by, the previous theorem, there exists a constant $0<K<\infty$ such that
\[
\|f\circ\varphi^{-1}\mid\overset{\circ}{W_{2}^{1}}(\mathbb{D})\|\leq K\|f\mid\overset{\circ}{W_{2}^{1}}(\Omega)\|
\]
 for any function $f\in\overset{\circ}{W_{2}^{1}}(\Omega)$.

By the classical Sobolev embedding theorem for the unit disc $\mathbb{D}$, the embedding operators 
$i_r:\overset{\circ}{W_{2}^{1}}(\mathbb{D})\hookrightarrow L_{r}(\mathbb{D})$ are bounded and compact for any $1\leq r<\infty$.

Using the change of variable formula and the boundedness of $i_r$, we obtain
\begin{multline*}
\|f\mid L_{r}(\Omega,h)\|=\biggl(\int\limits _{\Omega}|f(z)|^{r}J(z,\varphi)~d\mu\biggr)^{\frac{1}{r}}=\biggl(\int\limits _{\mathbb{D}}|f\circ\varphi^{-1}(w)|^{r}~d\mu\biggr)^{\frac{1}{r}}\\
=\|f\circ\varphi^{-1}\mid L_{r}(\mathbb{D})\|\leq \|i_r \| \|f\circ\varphi^{-1}\mid\overset{\circ}{W_{2}^{1}}(\mathbb{D})\|\leq \|i_r \|\cdot K\cdot\|f\mid\overset{\circ}{W_{2}^{1}}(\Omega)\|.
\end{multline*}

Therefore, the embedding operators $j_r:\overset{\circ}{W_{2}^{1}}(\Omega)\hookrightarrow L_{r}(\Omega,h)$ 
are compact as compositions of the compact embedding operators $i_r:\overset{\circ}{W_{2}^{1}}(\mathbb{D})\hookrightarrow L_{r}(\mathbb{D})$
and bounded composition operators. 
\end{proof}
The next theorem about composition operators was formulated and proved in \cite{U1,VU1} for Sobolev homeomorphisms. 
In the case of conformal homeomorphisms, this statement is much simpler. Here we reproduce a comparatively simple proof 
for the conformal case that enables us to avoid the main technicalities. As preliminary information, this proof uses
only the result about the existence of a bounded monotone countably additive function \cite{U1, VU1}:

\begin{thm}\label{thm:AddFun}
Suppose that a mapping $\varphi : \Omega\to \Omega^{\prime}$ induces a bounded composition operator
$$
\varphi^{\ast} : L^1_p(\Omega^{\prime})\to L^1_q(\Omega),\quad 1\leq q<p\leq\infty.
$$
Then
$$
\Phi(A^{\prime})=\sup\limits_{f\in L_{p}^{1}(A^{\prime})\cap C_0(A')}
\Biggl(
\frac{\bigl\|\varphi^{\ast} f\mid {L}_{q}^{1}(\Omega)\bigr\|}
{\bigl\|f\mid L_{p}^{1}(A^{\prime})\bigr\|}
\Biggr)^{\kappa},
$$
where the number $\kappa$ is defined by $1/\kappa = 1/q - 1/p$, 
is a bounded monotone countably additive function defined on
open bounded subsets $A^{\prime}\subset \Omega^{\prime}$.
\end{thm}

Using this ``localization principle'' for composition operators, we prove

\begin{thm}
\label{thm:CompNecSuf} Let $\Omega,\Omega'\subset\mathbb{C}$
be plane domains. A conformal homeomorphism $w=\varphi(z):\Omega\to\Omega'$
induces a bounded composition operator $\varphi^{\ast}:L_{p}^{1}(\Omega')\to L_{q}^{1}(\Omega)$,
$1\leq q<p<\infty$, if and only if 
\[
\biggl(\int\limits _{\Omega}|\varphi'(z)|^{\frac{(p-2)q}{p-q}}~d\mu\biggr)^{\frac{p-q}{pq}}=K<+\infty.
\]
\end{thm}
\begin{proof}
\noun{Necessity}. Suppose that the composition operator 
\[
\varphi^{\ast}:L_{p}^{1}(\Omega')\to L_{q}^{1}(\Omega),\,\,\,1\leq q<p<\infty,
\]
is bounded. Then by Theorem \ref{thm:AddFun} there exists a bounded monotone countably
additive function $\Phi$ defined on open bounded subsets 
$A'\subset\Omega'$ such that, for every function $f\in L_{p}^{1}(\Omega')\cap C_{0}(A')$,
\begin{equation}
\|\varphi^{\ast}(f)\mid L_{q}^{1}(\Omega)\|\leq\bigl(\Phi(A')\bigr)^{\frac{p-q}{pq}}\|f\mid L_{p}^{1}(\Omega)\|.
\end{equation}
Fix a cut function $\eta\in C_{0}^{\infty}(\mathbb{C})$ which
is equal to one on the set $\{w\in\mathbb C: |w|<1\}$ and is equal to zero outside of the set $\{w\in\mathbb C: |w|<2\}$.
Inserting the functions 
\[
f_{R}(w)=Re(w-w_{0})\eta\biggl(\frac{w-w_{0}}{r}\biggr),\,\,\, w_{0}\in\Omega',
\]
and
\[
f_{I}(w)=Im(w-w_{0})\eta\biggl(\frac{w-w_{0}}{r}\biggr),\,\,\, w_{0}\in\Omega',
\]
in this inequality, we obtain 
\[
\biggl(\int\limits _{\varphi^{-1}(B(w_{0},r))}|\varphi'(z)|^{q}~d\mu\biggr)^{\frac{1}{q}}\leq C\Phi(B(w_{0},2r))^{\frac{p-q}{pq}}|B(w_0,r)|^{\frac{1}{p}}
\]
when $B(w_{0},r)=\{w\in\mathbb C: |w-w_0|<r\}$, $B(w_0,2r)\subset\Omega'$.

By the change of variable formula, using the equality $|\varphi'(z)|^{2}=J(z,\varphi)$,
we infer 
\begin{multline*}
\int\limits _{\varphi^{-1}(B(w_{0},r))}|\varphi'(z)|^{q}~d\mu=\int\limits _{\varphi^{-1}(B(w_{0},r))}|\varphi'(z)|^{q-2}J(z,\varphi)~d\mu\\
=\int\limits _{B(w_{0},r)}|\varphi'(\varphi^{-1}(w))|^{q-2}~d\mu.
\end{multline*}
 Hence, 
\[
\biggl(\int\limits _{B(w_{0},r)}|\varphi'(\varphi^{-1}(w))|^{q-2}~d\mu\biggr)^{\frac{1}{q}}\leq C\Phi(B(w_{0},2r))^{\frac{p-q}{pq}}|B(w_0,r)|^{\frac{1}{p}}
\]
and we have 
\[
\frac{1}{|B(w_{0},r)|}\int\limits _{B(w_{0},r)}|\varphi'(\varphi^{-1}(w))|^{q-2}~d\mu\leq C\biggl(\frac{\Phi(B(w_{0},2r)}{|B(w_{0},r)|}\biggr)^{\frac{p-q}{p}}.
\]
Passing to the limit as $r\to0$, we get 
\[
|\varphi'(\varphi^{-1}(w))|^{q-2}\leq C\bigl(\Phi'(w)\bigr)^{\frac{p-q}{p}}\,\,\,\text{for almost all}\,\,\, w\in\Omega'.
\]
Hence, for every open bounded subset $V\subset\Omega'$
$$
\int\limits _{V}|\varphi'(\varphi^{-1}(w))|^{\frac{p(q-2)}{p-q}}~d\mu
\leq C\int\limits _{V}\Phi'(w)~d\mu\leq C\Phi(V)\leq C\|\varphi^{\ast}\|^{\frac{pq}{p-q}}.
$$
Since $V$ is an arbitrary subset of $\Omega'$ we have that
$$
\int\limits _{\Omega'}|\varphi'(\varphi^{-1}(w))|^{\frac{p(q-2)}{p-q}}~d\mu\leq C\|\varphi^{\ast}\|^{\frac{pq}{p-q}}.
$$

Therefore
\begin{multline*}
\int\limits _{\Omega}|\varphi'(z)|^{\frac{(p-2)q}{p-q}}~d\mu=\int\limits _{\Omega}|\varphi'(z)|^{\frac{p(q-2)}{p-q}}J(z,\varphi)~d\mu\\
=\int\limits _{\Omega'}|\varphi'(\varphi^{-1}(w))|^{\frac{p(q-2)}{p-q}}~d\mu
\leq C\|\varphi^{\ast}\|^{\frac{pq}{p-q}}.
\end{multline*}

\noun{Sufficiency.} Let $f\in L_{p}^{1}(\Omega')$. Then, since $\varphi$
is a smooth mapping, the composition $\varphi^{\ast}(f)=f\circ\varphi$
is defined almost everywhere and weakly differentiable in $\Omega$. Hence, using the equality $|\varphi'(z)|^{2}=J(z,\varphi)$,
we obtain 
\begin{multline*}
\|\varphi^{\ast}(f)\mid L_{q}^{1}(\Omega)\|=\biggl(\int\limits _{\Omega}|\nabla(f\circ\varphi(z))|^{q}~d\mu\biggr)^{\frac{1}{q}}=\biggl(\int\limits _{\Omega}|\nabla f|^{q}(\varphi(z))|\varphi'(z)|^{q}~d\mu\biggr)^{\frac{1}{q}}\\
=\biggl(\int\limits _{\Omega}|\nabla f|^{q}(\varphi(z))J(z,\varphi)^{\frac{q}{p}}|\varphi'(z)|^{\frac{(p-2)q}{p}}~d\mu\biggr)^{\frac{1}{q}}.
\end{multline*}
Thus, using H\"older's inequality, we have 
\begin{multline*}
\|\varphi^{\ast}(f)\mid L_{q}^{1}(\Omega)\|=\biggl(\int\limits _{\Omega}|\nabla f|^{q}(\varphi(z))J(z,\varphi)^{\frac{q}{p}}|\varphi'(z)|^{\frac{(p-2)q}{p}}~d\mu\biggr)^{\frac{1}{q}}\\
\leq\biggl(\int\limits _{\Omega}|\nabla f|^{p}(\varphi(z))J(z,\varphi)~d\mu\biggr)^{\frac{1}{p}}\biggl(\int\limits _{\Omega}|\varphi'(z)|^{\frac{(p-2)q}{p-q}}~d\mu\biggr)^{\frac{p-q}{pq}}=K\|f\mid L_{p}^{1}(\Omega')\|.
\end{multline*}
\end{proof}

\begin{thm}
\label{thm:InverseCompL} Let $\Omega\subset\mathbb{C}$ be a simply connected domain with non-empty boundary and $\varphi : \Omega\to\mathbb D$ be a conformal homeomorphism. Suppose that the Inverse Brennan's Conjecture holds for the interval $[\alpha_0,2/3)$ where $\alpha_{0}\in\big(-2,0\big)$ and $p\in\big({(|\alpha_0|+2)}/{(|\alpha_0|+1)},2\big)$.

Then the inverse
mapping $\varphi^{-1}$ induces a bounded composition operator 
\[
(\varphi^{-1})^{\ast}:{L_{p}^{1}}(\Omega)\to{L_{q}^{1}}(\mathbb{D})
\]
for any $q$ such that 
\[
1\leq q\leq {p\left|\alpha_{0}\right|}/{(2+\left|\alpha_{0}\right|-p)}<2p/(4-p).
\]
\end{thm}

\begin{proof}
By the Inverse Brennan's Conjecture 
\[
\int\limits _{\mathbb{D}}|(\varphi^{-1})'(w)|^{\alpha}~d\mu<+\infty,\quad\text{for all}\quad-2<\alpha<2/3.
\]

Choose $\alpha \in (-2,0)$. By Theorem \ref{thm:CompNecSuf}, the composition operator
\[
(\varphi^{-1})^{\ast}:L_{p}^{1}(\Omega)\to L_{q}^{1}(\mathbb{D})
\]
is bounded for all $1\leq q<p<2$ such that
$$
\frac{(p-2)q}{p-q}=\alpha.
$$

The inequality 
$$
\frac{(p-2)q}{p-q}<0
$$
holds for all $1<p<2$ and $1\leq q<p<2$.

Because $p<2$ the inequality 
$$
-2<\alpha=\frac{(p-2)q}{p-q}
$$
is equivalent to 
$$
q <\frac{2p}{4-p}<p.
$$
By Theorem \ref{thm:CompNecSuf} we have the restriction on $q$: $q\geq 1$. Hence 
$$
\frac{2p}{4-p}>1, \,\,  \text{i.~e.} \,\, p>\frac{4}{3}.
$$

It is the best possible estimate on $p$ under the assumption that the Inverse Brennan's conjecture is  correct. Because the conjecture is not proved completely, we  can suppose in the statement of this Theorem that it holds for some $-2< \alpha_0 <0$. In this case restrictions on $p$ will depend on $\alpha_0$.

The inequality 
$$
\alpha_0\leq\frac{(p-2)q}{p-q}
$$
is equivalent to 
$$
 1 \leq q \leq\frac{|\alpha_0|p}{2+|\alpha_0|-p}.
$$
Because the function $u(|\alpha_0|)=\frac{|\alpha_0|p}{2+|\alpha_0|-p}$ is an increasing function,  $|\alpha_0|<2$  and $p<2$ we obtain 
$$
q \leq \frac{|\alpha_0|p}{2+|\alpha_0|-p}<\frac{2p}{4-p}<p<2.
$$
Let us check for which numbers $p$ the condition $q\geq 1$  is correct, i.~e. for which numbers $p$ the following inequality 
$$
1< \frac{|\alpha_0|p}{2+|\alpha_0|-p}
$$
holds. By simple calculations we obtain that it holds for 
$$
p > \frac{|\alpha_0|+2}{|\alpha_0|+1}>\frac{4}{3}.
$$

 Two previous inequalities permit us to conclude that for any fixed $p\in({(|\alpha_0|+2)}/{(|\alpha_0|+1)},2)$ and for any $q$ such that 
$$
1\leq q \leq \frac{|\alpha_0|p}{2+|\alpha_0|-p}<\frac{2p}{4-p}<2
$$
the inequality
\[
\|\nabla(f\circ\varphi^{-1})\mid L_{q}(\mathbb{D})\|\leq K\| \nabla f\mid L_{p}(\Omega)\|
\]
holds for every 
function $f\in L_{p}^{1}(\Omega)$.
\end{proof}

\begin{prop}
\label{thm:PoincareCompSup} Suppose that $\Omega\subset\mathbb{C}$ is a simply connected domain with non empty boundary, the Inverse Brennan's Conjecture holds for the interval $[\alpha_0,2/3)$ where $\alpha_{0}\in\big(-2,0\big)$ and $h(z)=J(z,\varphi)$ is
the conformal weight defined by a conformal homeomorphism $\varphi : \Omega\to\mathbb D$. Then for every ${(|\alpha_0|+2)}/{(|\alpha_0|+1)}<p<2$ and every function $f\in C_{0}^{\infty}(\Omega)$, the inequality
\[
\|f\mid L_{r}(\Omega,h)\|\leq M\|\nabla f\mid L_{p}(\Omega)\|
\]
holds for any $r$ such that \[
1\leq r\leq \frac{2p}{2-p}\cdot\frac{\left|\alpha_{0}\right|}{2+\left|\alpha_{0}\right|}<\frac{p}{2-p}.
\]
The constant $M$ does not depend on~$f$.\end{prop}
\begin{proof}
By the Riemann Mapping Theorem, there exists a conformal homeomorphism
$w=\varphi(z):\Omega\to\mathbb{D}$, and by the Inverse Brennan's Conjecture, 
\[
\int\limits _{\mathbb{D}}|(\varphi^{-1})^{\prime}(w)|^{\alpha}~dw<+\infty,\quad\text{for all}\quad-2<\alpha_0<\alpha<2/3.
\]
Hence, by Theorem \ref{thm:InverseCompL}, 
the inequality
\[
\|\nabla(f\circ\varphi^{-1})\mid L_{q}(\mathbb{D})\|\leq K\| \nabla f\mid L_{p}(\Omega)\|
\]
holds for every 
function $f\in L_{p}^{1}(\Omega)$ and for any $q$ such that 
\[
1\leq q\leq {p\left|\alpha_{0}\right|}/{(2+\left|\alpha_{0}\right|-p)}<2p/(4-p).
\]

Choose arbitrarily $f\in C_{0}^{\infty}(\Omega)$. Then
  $f\circ\varphi^{-1}\in C_{0}^{\infty}(\mathbb{D})$ and, by the classical Poincar\'e-Sobolev inequality,
\begin{equation}
\|f\circ\varphi^{-1}\mid L_{r}(\mathbb{D})\|\leq A\|\nabla (f\circ\varphi^{-1})\mid L_{q}(\mathbb{D})\| \label{eq:PS}
\end{equation}
for any $r$ such that 
\[
1\leq r\leq \frac{2q}{2-q}
\]
Combining inequalities for $q$ and $r$ we conclude that the inequality (\ref{eq:PS}) holds for any $r$ such that
\[
1\leq r\leq \frac{2p}{2-p}\cdot\frac{\left|\alpha_{0}\right|}{2+\left|\alpha_{0}\right|}<\frac{p}{2-p}.
\]

Using the change of variable formula, we finally infer
\begin{multline*}
\|f\mid L_{r}(\Omega,h)\|=\biggl(\int\limits _{\Omega}|f(z)|^{r}J(z,\varphi)~d\mu\biggr)^{\frac{1}{r}}=\biggl(\int\limits _{\mathbb{D}}|f(\varphi^{-1}(w))|^{r}~d\mu\biggr)^{\frac{1}{r}}\\
\\
=\|f\circ\varphi^{-1}\mid L_{r}(\mathbb{D})\|\leq A\|\nabla (f\circ\varphi^{-1})\mid L_{q}(\mathbb{D})\|\leq AK\|\nabla f\mid L_{p}(\Omega)\|.
\end{multline*}
\end{proof}

\begin{thm}
\label{thm:InverseComp} Let $\Omega\subset\mathbb{C}$ be a simply connected domain with non-empty boundary and $\varphi : \Omega\to\mathbb D$ be a conformal homeomorphism. Suppose that the Inverse Brennan's Conjecture holds for the interval $[\alpha_0,2/3)$ where $\alpha_{0}\in\big(-2,0\big)$ and $p\in\big({(|\alpha_0|+2)}/{(|\alpha_0|+1)},2\big)$.

Then the inverse
mapping $\varphi^{-1}$ induces a bounded composition operator 
\[
(\varphi^{-1})^{\ast}:\overset{\circ}{W_{p}^{1}}(\Omega)\to\overset{\circ}{W_{q}^{1}}(\mathbb{D})
\]
for any $q$ such that 
\[
1\leq q\leq {p\left|\alpha_{0}\right|}/{(2+\left|\alpha_{0}\right|-p)}<2p/(4-p).
\]
\end{thm}

{\bf Remark. } Of course, we can choose the best known estimate $\alpha_{0}>-2$ in the Inverse Brennan's
Conjecture.
\begin{proof}
By the Inverse Brennan's Conjecture, we have 
\[
\int\limits _{\mathbb{D}}|(\varphi^{-1})^{\prime}(w)|^{\alpha}~d\mu<+\infty,\quad\text{for all}\quad -2<\alpha_{0}<\alpha<2/3.
\]
Let $f\in \overset{\circ}{W_{p}^{1}}(\Omega)$ be a smooth function with a compact support. By Theorem \ref{thm:InverseCompL}, for any function $f\in C_{0}^{\infty}(\Omega)$,
the inequality 
\[
\|f\circ\varphi^{-1}\mid L_{q}^{1}(\mathbb{D})\|\leq C\|f\mid L_{p}^{1}(\Omega)\|
\]
holds for all $q: 1\leq q\leq {p\left|\alpha_{0}\right|}/({2+\left|\alpha_{0}\right|-p)}<2p/(4-p)<2$. 
Here the constant $0<C<\infty$ is independent of~$f$.

Since $f\in C_{0}^{\infty}(\Omega)$, its $\mathrm{supp}(f)$
is compact and its image $\varphi(\mathrm{supp}(f))$ is compact too. The following estimate holds
for $\|f\circ\varphi^{-1}|L_{q}(\mathbb{D})\|$:
\begin{multline*}
\|f\circ\varphi^{-1}|L_{q}(\mathbb{D})\|=\biggl(\int\limits _{\varphi(\mathrm{supp}(f))}|f\circ\varphi^{-1}|^{q}~d\mu\biggr)^{\frac{1}{q}}\\
=\biggl(\int\limits _{\varphi(\mathrm{supp}(f))}|f\circ\varphi^{-1}|^{q}J(w,\varphi^{-1})^{\frac{q}{p}}\frac{1}{J(w,\varphi^{-1})^{\frac{q}{p}}}~d\mu\biggr)^{\frac{1}{q}}\\
\leq\biggl(\int\limits _{\varphi(\mathrm{supp}(f))}J(w,\varphi^{-1})^{\frac{q}{q-p}}~d\mu\biggr)^{\frac{p-q}{pq}}\biggl(\int\limits _{\varphi(\mathrm{supp}(f))}|f\circ\varphi^{-1}|^{p}J(w,\varphi^{-1})~d\mu\biggr)^{\frac{1}{p}}.
\end{multline*}

Putting 
\[
Q(\varphi,f):=\biggl(\int\limits _{\varphi(\mathrm{supp}(f))}J(w,\varphi^{-1})^{\frac{q}{q-p}}~d\mu\biggr)^{\frac{p-q}{pq}},
\]
we conclude that 
\[
\|f\circ\varphi^{-1}|L_{q}(\mathbb{D})\|\leq Q(\varphi,f)\|f|L_{p}(\Omega)\|.
\]

Thus, the function $f\circ\varphi^{-1}$ belongs to the Sobolev space
$\overset{\circ}{W_{q}^{1}}(\mathbb{D})$.
Using the Poincar\'e-Sobolev inequality for the unit disc 
\[
\|f\circ\varphi^{-1}\mid L_{q}(\mathbb{D})\|\leq A\|\nabla (f\circ\varphi^{-1})\mid L_{q}(\mathbb{D})\|
\]
we obtain the inequality 
\begin{multline*}
\|f\circ\varphi^{-1}\mid W_{q}^{1}(\mathbb{D})\|=\|f\circ\varphi^{-1}\mid L_{q}(\mathbb{D})\|+\|\nabla(f\circ\varphi^{-1})\mid L_{q}(\mathbb{D})\|\\
\leq(A+1)\|\nabla(f\circ\varphi^{-1})\mid L_{q}(\mathbb{D})\|\leq(A+1)C\|f\mid L_{p}^{1}(\Omega)\|\leq K\|f\mid W_{p}^{1}(\Omega)\|
\end{multline*}
that holds for every $f\in C_{0}^{\infty}(\Omega)$.

Using the density of smooth functions with compact supports in $f\in\overset{\circ}{W_{p}^{1}}(\Omega)$,
we can extend the last inequality to arbitrary $f\in\overset{\circ}{W_{p}^{1}}(\Omega)$
(see, for example, \cite{VU1, GMU}). It allows us to conclude
 finally that the composition operators
\[
(\varphi^{-1})^{\ast}:\overset{\circ}{W_{p}^{1}}(\Omega)\to\overset{\circ}{W_{q}^{1}}(\mathbb{D})
\]
are bounded for all $q$ satisfying the conditions $1\leq q\leq {p\left|\alpha_{0}\right|}/({2+\left|\alpha_{0}\right|-p)}<2p/(4-p)$. \end{proof}

\begin{thm}
\label{thm:WeightEmb} Suppose that $\Omega\subset\mathbb{C}$ is a simply connected domain with non empty boundary, the Inverse Brennan's Conjecture holds for the interval $[\alpha_0,2/3)$ where $\alpha_{0}\in\big(-2,0\big)$, $p\in\big({(|\alpha_0|+2)}/{(|\alpha_0|+1)},2\big)$  and $h(z)=J(z,\varphi)$ is the conformal weight defined by a conformal mapping $\varphi : \Omega\to\mathbb D$.   

Then the embedding operator
\[
j_r:\overset{\circ}{W_{p}^{1}}(\Omega)\hookrightarrow L_{r}(\Omega,h)
\]
is compact for every $r$ such that 
\[
1\leq r< \frac{2p}{2-p}\cdot\frac{\left|\alpha_{0}\right|}{2+\left|\alpha_{0}\right|}<\frac{p}{2-p}.
\]
\end{thm}
{\bf Remark. } Of course, we can choose the best known estimate $\alpha_{0}>-2$ in the Inverse Brennan's
Conjecture.
\begin{proof}
Choose an~arbitrary function $f\in\overset{\circ}{W_{p}^{1}}(\Omega)$. Let $w=\varphi(z): \Omega\to\mathbb D$ 
be a conformal homeomorphism. By the previous theorem, $(\varphi^{-1})^{\ast}f\in\overset{\circ}{W_{q}^{1}}(\mathbb{D})$
for $1\leq q\leq {p\left|\alpha_{0}\right|}/({2+\left|\alpha_{0}\right|-p)}<2p/(4-p)<2$ and 
\[
\|(\varphi^{-1})^{\ast}f\mid\overset{\circ}{W_{q}^{1}}(\mathbb{D})\|\leq K\|f\mid\overset{\circ}{W_{p}^{1}}(\Omega)\|
\]
where $K<\infty$ does not depend on~$f$. By the classical embedding theorem,
the function $(\varphi^{-1})^{\ast}f\in L_{r}(\mathbb{D})$ for
$r : 1\leq r<{2p\left|\alpha_{0}\right|}/{(2\left|\alpha_{0}\right|+4-p(2+\left|\alpha_{0}\right|))}<{p}/{(2-p)}$
and the inequality 
\[
\|(\varphi^{-1})^{\ast}f\mid L_{r}(\mathbb{D})\|\leq A\|(\varphi^{-1})^{\ast}f\mid\overset{\circ}{W_{q}^{1}}(\mathbb{D})\|
\]
holds for any $f\in \overset{\circ}{W_{q}^{1}}(\mathbb{D})$.
Using the change of variable formula, we obtain 
\begin{multline*}
\|f\mid L_{r}(\Omega,h)\|=\biggl(\int\limits _{\Omega}|f(z)|^{r}J(z,\varphi)~d\mu\biggr)^{\frac{1}{r}}=\biggl(\int\limits _{\mathbb{D}}|f(\varphi^{-1}(w))|^{r}~d\mu\biggr)^{\frac{1}{r}}\\
\\
=\|(\varphi^{-1})^{\ast}f\mid L_{r}(\mathbb{D})\|\leq A\|(\varphi^{-1})^{\ast}f\mid\overset{\circ}{W_{q}^{1}}(\mathbb{D})\|\leq AK\|f\mid\overset{\circ}{W_{p}^{1}}(\Omega)\|.
\end{multline*}
Consequently, the embedding operator 
\[
j_r:\overset{\circ}{W_{p}^{1}}(\Omega)\hookrightarrow L_{r}(\Omega,h)
\]
is compact as the composition of the bounded composition operator
$$
(\varphi^{-1})^{\ast} : \overset{\circ}{W^1_p}(\Omega)\to \overset{\circ}{W^1_1}(\mathbb D),
$$
compact embedding operator 
$$
i_r: \overset{\circ}{W_{q}^{1}}(\mathbb{D})\hookrightarrow L_{r}(\mathbb{D})
$$
and the bounded composition operator 
$$
\varphi^{\ast} : L_r(\mathbb D) \to L_r(\Omega,h).
$$
\end{proof}

\section{Applications to Elliptic Equations}

As a standard application, we prove the solvability of the classical Dirichlet
problem for the degenerate Laplace operator on an arbitrary simply connected
plane domain $\Omega\subset\mathbb{C}$ with non-empty boundary.

Define the weighted Sobolev space $W_{p}^{1}(\Omega,h,1)$, $1\leq p<\infty$, 
as the normed space of locally integrable weakly differentiable functions
$f:\Omega\to\mathbb{R}$ equipped with the following norm: 
\[
\|f\mid W_{p}^{1}(\Omega)\|=\biggr(\int\limits _{\Omega}|f|^{p}(z)\,h(z)\, d\mu\biggr)^{1/p}+\biggr(\int\limits _{\Omega}|\nabla f(z)|^{p} \, d\mu\biggr)^{1/p}.
\]
Here $h$ is the universal conformal weight.

Such type weighted spaces were introduced and investigated in \cite{GMU} for $\mathcal{V}_p$-weights.

The Sobolev space $\overset{\circ}{W_{p}^{1}}(\Omega,h,1)$, $1\leq p<\infty$, 
is defined as the closure of the space $C_{0}^{\infty}(\Omega)$  of smooth functions with compact
support in the norm of $W_{p}^{1}(\Omega,h,1)$.

We shall use a short notation for the inner products: $<u,v>:=\int_\Omega u(z)v(z)d\mu $ 
in $L_2(\Omega)$ and  $<u,v>_h:=\int_\Omega u(z)v(z)h(z)d\mu $ in $L_2(\Omega,h)$ and also the notation
$[u,v]:=<\nabla u, \nabla v>$. 

The problem is as follows: 
\begin{gather} \label{Dir}
\Delta u=fh\,\, \text{in}\,\,\Omega,\\
u\vert_{\partial\Omega}=0. \label{Dir1}
\end{gather}

The weak statement of this Dirichlet problem is as follows:

A function $u$ solves the previous problem iff $u\in \overset{\circ}{W_{2}^{1}}(\Omega,h,1) $
and $$[u,v]=<\nabla u,\nabla v>=\int_\Omega f(z)v(z)h(z)d\mu$$ for all
$v\in \overset{\circ}{W_{2}^{1}}(\Omega,h,1)$, under provided that this integral makes sense.

We prove that, for any simply connected plane domain with non-empty boundary and for any $f\in L_{p}(\Omega,h)$, 
$1<p<\infty$, there exists a~unique (weak) solution to the problem (\ref{Dir},\ref{Dir1}).

\begin{thm}
\label{thm:DirProb}Let $\Omega$ be a simply connected plane domain
with non-empty boundary and let $1<p<\infty$. If $f\in L_{p}(\Omega,h)$
then there exists the unique weak solution $u\in\overset{\circ}{W}_{2}^{1}(\Omega,h,1)$
of the problem (\ref{Dir},\ref{Dir1}). \end{thm}
\begin{proof}
The function  $f\in L_{p}(\Omega,h)$ induces a linear functional
$F:C_{0}^{\infty}(\Omega)\rightarrow \mathbb R$ by the standard rule 
\[
F(v)=\int\limits _{\Omega}f(z)v(z)h(z)~d\mu,\,\,\, v\in C_{0}^{\infty}(\Omega).
\]
By Theorem \ref{thm:Weightem}: 
\[
\|v|L_{2}(\Omega,h)\|\leq K\|\nabla v|L_{2}(\Omega)\|.
\]

The last inequality shows that the norm $[u,u]^{\frac{1}{2}}:=\|\nabla u|L_{2}(\Omega)\|$ 
is equivalent to the $W^{1}_2(\Omega,h,1)$-norm on $\overset{\circ}{W}_{2}^{1}(\Omega,h,1)$
 and the corresponding inner product $[u,v]$ induces a Hilbert structure on $\overset{\circ}{W_{2}^{1}}(\Omega,h,1)$. 

Using Theorem \ref{thm:Weightem} and the~density of $C_0^{\infty}(\Omega)$ in $\overset{\circ}{W}_{2}^{1}(\Omega,h,1)$,
we show that $F$ is a bounded linear functional in $\overset{\circ}{W_{2}^{1}}(\Omega,h,1)$.
 
For any $v \in C_{0}^{\infty}(\Omega)$ by Theorem \ref{thm:Weightem} we have
\begin{multline*}
|F(v)|\leq\|f\cdot v\mid L_{1}(\Omega,h)\|\\
\leq\|f\mid L_{p}(\Omega,h)\|\cdot\|v\mid L_{p'}(\Omega,h)\|\leq C\|f\mid L_{p}(\Omega,h)\|\cdot\|v\mid\overset{\circ}{W}_{2}^{1}(\Omega,h,1)\|,
\end{multline*}
for any $1\leq p < \infty$, $p'={p}/{(p-1)}$.

Therefore, the bounded functional $F$ defined on the dense subset $C_{0}^{\infty}(\Omega)$ can be extended 
to $\overset{\circ}{W}_{2}^{1}(\Omega,h,1)$.

Hence, by the Riesz theorem about linear functionals in a Hilbert space,  
$[u,v]=[Bf,v]$ for all $v \in \overset{\circ}{W}_{2}^{1}(\Omega,h,1)$,
where $B$ is a bounded linear operator from $L^{p}(\Omega,h)$ into $\overset{\circ}{W}_{2}^{1}(\Omega,h,1)$.
Thus, $u=Bf$ is a~unique solution for (\ref{Dir},\ref{Dir1}).
\end{proof}

\noindent
Vladimir Gol'dshtein  \,  \hskip 3.2cm Alexander Ukhlov

\noindent
Department of Mathematics   \hskip 2.25cm Department of Mathematics

\noindent
Ben-Gurion University of the Negev  \hskip 1.05cm Ben-Gurion University of the Negev

\noindent
P.O.Box 653, Beer Sheva, 84105, Israel  \hskip 0.7cm P.O.Box 653, Beer Sheva, 84105, Israel

\noindent
E-mail: vladimir@bgu.ac.il  \hskip 2.5cm E-mail: ukhlov@math.bgu.ac.il

\end{document}